\documentclass[a4paper,10pt]{article}
\usepackage{amsmath}
\usepackage{amsthm}
\usepackage{amssymb}
\usepackage[latin1]{inputenc}
\usepackage[english]{babel}

\theoremstyle{plain}
\newtheorem{thm}{Theorem}[section]
\newtheorem{prop}[thm]{Proposition}
\newtheorem{lem}[thm]{Lemma}
\newtheorem{cor}[thm]{Corollary}
\theoremstyle{definition}
\newtheorem{defn}[thm]{Definition}
\theoremstyle{remark}

\newtheorem{rem}[thm]{Remark}

\title{On the K-stability of complete intersections in polarized manifolds}
\author{Claudio Arezzo and Alberto Della Vedova \footnote{{\bf Address}: Dipartimento di Matematica, Universit\`a degli Studi di Parma, Viale G. P. Usberti, 53/A - 43100 Parma (Italy). {\bf E-mail}: claudio.arezzo@unipr.it, alberto.dellavedova@unipr.it} }
\date{}

\begin{document}
\maketitle

\begin{abstract} \noindent We consider the problem of existence of constant scalar curvature K\"ahler metrics on complete intersections of sections of vector bundles. In particular we give general formulas relating the Futaki invariant of such a manifold to the weight of sections defining it and to the Futaki invariant of the ambient manifold. As applications we give a new Mukai-Umemura-Tian like example of Fano 5-fold admitting no K\"ahler-Einstein metric and a strong evidence of $K$-stability of complete intersections on Grassmannians.

\medskip 

 \end{abstract}

\section{Introduction}

\noindent The problem of determining which manifolds admit a K\"ahler constant scalar curvature metric (Kcsc), and in which K\"ahler classes, is by now a central one in differential geometry 
and it has been approached with a variety of geometric and analytical methods.

\noindent A classical result due to Matsushima and Lichnerovicz \cite{ma,li} shows that a 
such a manifolds has a reductive identity component of the automorphisms group, a condition unsensitive of the K\"ahler class where we look for the Kcsc metric.
In the eighties Futaki \cite{fu}, later generalized by Calabi \cite{ca}, introduced an invariant, since then called {\em{the Futaki invariant}}, sensitive of the K\"ahler class.
The deep nature of this invariant has stimulated a great amount of research. While it can be used directly to show that a manifold $M$ does not have a Kcsc metric in a K\"ahler class, a more refined analysis, mainly due to Ding-Tian \cite{DT92}, Tian \cite{T97}, Paul-Tian \cite{PT} and Donaldson \cite{D02}, has led to relate this invariant on a manifold $M$ to the existence 
of Kcsc metric on {\em{any}} manifold degenerating in a  suitable sense to $M$.

\noindent This idea has been formalized in a precise conjecture due to Donaldson \cite{D02} relating the existence of such metrics to the $K$-stability of the polarized manifold.
We will summarize this in Section \ref{sec::pre}. The key point relevant for our paper is that the knowledge of the Futaki invariant gives informations on the existence of Kcsc metrics on the manifolds on which the calculations are carried on {\em{and also}} on any K\"ahler
manifold degenerating on it.

\noindent The problem of calculating explicitely the Futaki invariant of a polarized manifold has then got further importance. Its original analytical definition is extremely hard to use, since requires an explicit knowledge of the Ricci potential and of the K\"ahler metric, data  which are almost always missing. On the other hand it led to the discovery of the so called {\em{localization formulae}} \cite{fu2, T1} which have been a very useful tool
in this problem. Yet, they require an explicit knowledge of the space of holomorphic vector fields and of the K\"ahler metric which is again very hard to have.

\noindent Finally Donaldson \cite{D02} gave a pure cohomological interpretation of the Futaki invariant, extending it to singular varieties and schemes, which is the one we use in this paper and that will be recalled in Section \ref{sec::pre}. Let us just recall at this point that the Futaki invariant is defined for a polarized 
scheme $(M,L)$ endowed with a 
$\mathbb C^\times$-action $\rho \colon \mathbb C^\times \rightarrow {\rm Aut}(M)$ that linearizes on $L$ (hence a holomorphic vector field $\eta_{\rho}$). We will then denote thorought 
this paper such a structure by $(M,L,\rho)$ and by $F(M,L,\rho)$ the Futaki invariant 
of $\eta_{\rho}$ in the class $c_1(L)$ of this triple.

\noindent We can now describe our result. We assume that we are given a polarized variety $(M,L)$ endowed with a $\mathbb C^\times$-action that linearizes on $L$. If $X\subset M$ is an invariant complete intersection of sections of holomorphic vector bundles $E_1,\dots,E_s$ on $M$, we will show that is possible to express $F(X,L_{|{X}},\rho)$ in terms of the weights of sections defining $X$ and holomorphic invariants of the bundles $E_j$'s and $L$.

\noindent  In this paper we make explicit the formula in two relevant cases: the first, when $L$ is the anti-canonical bundle $K_M^{-1}$ of $M$ and all $E_j$'s are isomorphic to a fixed vector bundle $E$ such that $\det E$ is a (rational) multiple of $L$ as linearized vector bundle; the second, when each $E_j$ is isomorphic to some power $L^{r_j}$ of the polarizing line bundle. We do not state the formula for the general case, but it can be recovered through some calculations from lemmata \ref{lem::asym_h^0_vect} and \ref{lem::asym_w^0_vect}.

\noindent  Let us consider the first case. Let $E$ be a $\mathbb C^\times$-linearized holomorphic vector bundle on a smooth Fano manifold $M$ such that $(\det E)^q = K_M^{-p}$  for some integers $p,q$. For each $j\in \{ 1,\dots,s \}$ let $\sigma_j \in H^0(M,E)$ be a non-zero holomorphic semi-invariant section, in other words there exists $\alpha_j \in \mathbb Z$ such that $\rho(t) \cdot \sigma_j = t^{\alpha_j}\sigma_j$. Thus the zero locus $X_j=\sigma_j^{-1}(0)$ is $\rho$ - invariant and $L=\det E$ restrict to a linearized ample line bundle on $X_j$. Consider the intersection $X = \bigcap_{j=1}^s X_j$ and assume that $\dim (X) = n- s k$, being $k = {\rm rank}(E)$.  Moreover, by adjunction, $X$ is a possibly singular Fano variety if $q-ps>0$. Our first result is the following 

\begin{thm}\label{thm::main_vect_intro}
Under the above conventions and assumptions we have
\begin{equation}\label{eq::F_E_j_intro}
F(X,L|_X, \rho) = \frac{ps-q}{2q} \frac{a_0(X,L|_X)}{d_0(X,L|_X)} - \frac{k}{2} \sum_{j=1}^s \alpha_j,
\end{equation}
where $d_0(X,L|_X)$ and $a_0(X,L|_X)$ are respectively the degree of $(X,L|_X)$ and its equivariant analogue (see definition \ref{defn::DF_invariant}) and can be computed by means of holomorphic invariants of $E$ and the quantity $\sum_{j=1}^s \alpha_j$. 
\end{thm}

\noindent The above theorem gives a significant simplification of the Donaldson version of the Futaki invariant (definition \ref{defn::DF_invariant}) in that the above formula involves only $a_0$ and $d_0$ and not $a_1$ and $d_1$ which are in general much harder to compute.

\noindent It is also important to notice that $\sum_{j=1}^s \alpha_j$ is nothing but the Mumford weight of the plane $P={\rm span}\{\sigma_j\} \in Gr(s, H^0(M,E))$. With an additional hypothesis on the linearization of the given $\mathbb C^\times$--action on $E$, theorem above gives the following 

\begin{cor}
Under the above conventions and assumptions, if the $\mathbb C^\times$-linearization on $E$ satisfies $ \int_M c_k^G(E)^sc_1^G(E)^{n-sk+1} = 0$, then
\begin{equation}
F(X,K_X^{-1},\rho) = -CT \sum_{j=1}^s \alpha_j,
\end{equation}
where $$C=\left(2p(n-sk+1)\int_M c_k(E)^sc_1(E)^{n-sk}\right)^{-1}>0 $$ and 
\begin{multline*} T = kp(n-sk+1)\int_M c_k(E)^sc_1(E)^{n-sk} \\ - (q-ps) \int_M c_k(E)^{s-1}c_{k-1}(E)c_1(E)^{n-sk+1} \end{multline*}
are characteristic numbers of $E$ (independent of the $\mathbb C^\times$--linearization). 
\end{cor}

\noindent The interest in the above Corollary is twofold. On the one hand it relates two very natural, and a priori unrelated, invariants of the manifold $X$ in a completely general setting. On the other hand it generalizes a special case, proved by completely different ad hoc arguments
by Tian \cite{T97}, used to produce the first (and up to now the only) examples of smooth Fano manifolds with discrete automorphism group without  K\"ahler-Einstein metrics.

\noindent  Another application of our study is that if $(M,L)$ is a complex Grassmannian anticanonically polarized and $P$ is a {\em{generic}} subspace of $H^0(M,E)$, in a sense explained in Section \ref{sec::appex}, then $X_P$ degenerates onto a $X_{P_{0}}$ whose Futaki invariant is positive, hence hinting at the $K$-stability of this type of manifolds. In particular this gives strong evidence to $K$-stability of these manifolds if their moduli space is discrete.

\noindent Of course the above Corollary rises the question whether $T$ has a specific sign.
We do not believe in general this to be the case, but we describe some classes of examples for which we can conclude, thanks to a theorem of of Beltrametti, Schneider and Sommese \cite{BSS96}, that $T$ is indeed positive (see also Remark \ref{rem::positivity_T}).

\noindent  Our second type of results comes from looking at classes different form the canonical one. We will restrict ourselves to the case when the bundles where to choose the sections are all line bundles and are all (possibly varying) powers of a fixed line bunlde $L$. Thus if $L$ is sufficiently positive we can embed $M$ in a projective space $\mathbb P^N$ and $X$ is the intersection of $M$ with a number of hypersurfaces. We are then interpreting our results in terms of Kcsc metrics in $c_1(L)$. This situation has been previously studied by Lu \cite{Lu} in the case when the ambient manifold is projective space. Again our result has a computational interest in that it makes 
very easy to calculate the Futaki invariant for  a great variety of manifolds, but also a 
conceptual one that we underline in the following

\begin{cor}
Let $(M,L)$ be a $n$-dimensional polarized manifold endowed with a $\mathbb C^\times$-action $\rho: \mathbb C^\times \to {\rm Aut}(M)$ and a linearization on $L$. For each $j\in\{1,\dots,s\}$ consider a section $\sigma_j \in H^0(M,L^{r})$ such that $\rho(t) \cdot \sigma_j = t^{\alpha_j} \sigma_j$ for some $\alpha_j \in \mathbb Z$. Let $X=\bigcap_{j=1}^s \sigma^{-1}(0)$. Suppose $\dim(X) = n-s$, then 

 \begin{eqnarray*} F(X,L|_X,\rho) &=&  F(M,L,\rho) - C \mu \left(X,M,L\right), \end{eqnarray*} where
 $\mu \left(X,M,L\right),$ is the Chow weight of the polarized manifold (see Section \ref{sec::Lrj} for the definition), and  $C \geq 0$ with equality if and only if $M\simeq \mathbb P^n$ and $r=1$.
 
\noindent In particular, if $M$ has a Kcsc in $c_1(L)$ and $X$ is $K$-semistable, then $(X,L_{|_{X}})$ is Chow stable.
\end{cor}
 
\noindent The relevance of this last statement is that the conclusion is {\em{not about asymptotic Chow stability}}, which is known to be related by a result of Donaldson \cite{D00} to the existence of Kcsc metrics. For example, even in the very special case of hypersurfaces of projective spaces, this gives strong further evidence of their $K$-semistability (cfr. Tian \cite{T94}).

\noindent Having dropped the assumption on the smoothness of $X$ we can use our formulae for singular varieties which arise as central fiber of test configurations. We give in Section \ref{sec::appex} an explicit example of this situation with a central fiber of our type with non positive Futaki invariant, hence producing non Kcsc manifolds (the degenerating ones).

\noindent Another explicit application of our formulae comes when looking at the quintic Del Pezzo threefold, $X_5$, for which it was not known whether it admits a Kcsc metric. In fact our analysis shows that it is $K$-stable, when confining to those test configurations whose central fibers are still 
manifolds of the type considered in our paper. While we believe a complete algebraic proof of its $K$-stability is then at hand, showing that every test configuration is indeed of this type, 
we remark that we can adapt a very recent 
observation of Donaldson \cite{D07} about the Mukai-Umemura threefold, to prove that this 
manifold (which is rigid in moduli) indeed  has a K\"ahler-Einstein metric.

\noindent Unfortunately the other Fano threefolds with ${\rm Pic} = \mathbb Z$ for which the existence of a canonical metric is unknown, when smooth do not have continuous automorphisms. If we take singular ones defined by sections of the appropriate bundles
with non positive Futaki invariant, we still cannot find test configurations with smooth general fibers. We leave this important problem for further research.

\medskip

\noindent Part of this work has been carried out in Fall 2007 during the visit of the second author at the Princeton University, whose hospitality is gratefully acknowledged. It is a great pleasure to thank G. Tian for many enlightening discussions. Thanks also to Y. Rubinstein and J. Stoppa  for many important conversations. 
 

\section{Preliminaries}\label{sec::pre}

At this point we recall some definitions (mainly form \cite{D02}) for future reference.

\begin{defn}\label{defn::DF_invariant}
Let $(V,L)$ be a $n$-dimensional polarized variety or scheme. Given a one parameter subgroup $\rho: \mathbb C^\times \to {\rm Aut}(V)$ with a linearization on $L$ and denoted by $w(V,L)$ the weight of the $\mathbb C^\times$-action induced on $\bigwedge^{\rm top} H^0(V,L)$, we have the following asymptotic expansions as $k \gg 0$:
\begin{eqnarray}
\label{eq::asym_h^0(V,L^k)} w(V,L^m) &=& a_0(V,L) m^{n+1} + a_1(V,L) m^{n} + O(m^{n-1}) \\
\label{eq::asym_w(V,L^k)}h^0(V,L^m) &=& d_0(V,L) m^n + d_1(V,L) m^{n-1} + O(m^{n-2})
\end{eqnarray}
The (normalized) \emph {Futaki invariant} of the action is $$ F(V,L, \rho) = \frac{a_0(V,L) \,d_1(V,L)}{d_0(V,L)^2} - \frac{a_1(V,L)}{d_0(V,L)}. $$
\end{defn}

\begin{rem}\label{rem::aclb}
Is not difficult to see that the Futaki invariant is unchanged if we replace $L$ with some tensor power $L^r$, moreover it is independent of the linearization chosen on $L$. Unlike the general case, when $V$ is smooth and $L=K_V^{-1}$ is the canonical bundle there is a natural linearization of the $\mathbb C^\times$-action $\rho$ on $L$ induced by the (holomorphic) tangent map $$d\rho: TM \to TM.$$ In this case we will call $L$ the \emph{anti-canonical linearized bundle}. 

We observe that the Futaki invariant of a polarized manifold $(V,L)$ assume a simple form when $L$ is the anti-canonical linearized line bundle. Indeed, by the equivariant Riemann-Roch theorem we get $d_0(V,K_V^{-1})=\int_V\frac{c_1(V)^n}{n!}$, $d_1(V,K_V^{-1}) = \int_V \frac{c_1(V)^n}{2(n-1)!}$, $a_0(V,K_V^{-1})=\int_V\frac{c_1^G(V)^{n+1}}{(n+1)!}$, $a_1(V,K_V^{-1})=\int_V\frac{c_1^G(V)^{n+1}}{2\,n!}$ (where $c_1^G$ denote the equivariant first Chern class), whence

\begin{equation*}
F(V, K_V^{-1}, \rho) = -\frac{1}{2} \frac{a_0(V,L)}{d_0(V,L)}.
\end{equation*}
\end{rem}

The relevance of the Futaki invariant is related to the definition of $K$-stability. To introduce it we need the following 

\begin{defn}
A \emph{test configuration} of a polarized manifold $(X,L)$ consists of a polarized scheme $(\mathcal X, \mathcal L)$ endowed with a $\mathbb C^\times$-action that linearizes on $\mathcal L$ and a flat $\mathbb C^\times$-equivariant map $\pi: \mathcal X \to \mathbb C$ such that $\mathcal L|_{\pi^{-1}(0)}$ is ample on $\pi^{-1}(0)$ and we have $(\pi^{-1}(1) , \mathcal L|_{\pi^{-1}(1)}) \simeq (X,L^r)$ for some $r>0$.

When $(X,L)$ has a $\mathbb C^t\times$ action $\rho: \mathbb C^\times \to {\rm Aut}(M)$, a test configuration where $\mathcal X = X \times \mathbb C$ and $\mathbb C^\times$ acts on $\mathcal X$ diagonally trought $\rho$ is called \emph{product configuration}.    
\end{defn}

\begin{defn}
The pair $(X,L)$ is \emph{$K$-stable} if for each test configuration for $(X,L)$ the Futaki invariant of the induced action on $(\pi^{-1}(0), \mathcal L|_{\pi^{-1}(0)})$ is greater than or equal to zero, with equality if and only if we have a product configuration.
\end{defn}

Finally we remark that the apparently different definition of $K$-stability given in \cite{D02} is due to the different choice of the sign in the dfinition of the Futaki invariant. 
 

\section{The case $(\det E)^q\simeq K_M^{-p}$}

\begin{thm}\label{thm::main_vect}
Let $(M,L)$, with $L=K^{-1}_M$, be a $n$-dimensional anti-canonically polarized Fano manifold endowed with a $\mathbb C^\times$-action $\rho: \mathbb C^\times \to {\rm Aut}(M)$ and a linearization on $L$. Let $E$ be a rank $k$ linearized vector bundle on $M$ such that $(\det E)^q \simeq L^p$ as linearized bundles for some $p$. For each $j\in\{1,\dots,s\}$ consider a non-zero section $\sigma_j \in H^0(M,E)$ such that $\rho(t) \cdot \sigma_j = t^{\alpha_j} \sigma_j$ for some $\alpha_j \in \mathbb Z$ and set $X=\bigcap_{j=1}^s \sigma_j^{-1}(0)$. If $\dim(X)=n-sk$, then we have
\begin{equation}\label{eq::F_E_j}
F(X,L|_X, \rho) = \frac{ps-q}{2q} \frac{a_0(X,L|_X)}{d_0(X,L|_X)} - \frac{k}{2} \sum_{j=1}^s \alpha_j,
\end{equation}
where
\begin{eqnarray*}
a_0(X,L|_X) &=& \int_M \frac{c_k^G(E)^s c_1^G(L)^{n-sk+1}}{(n-sk+1)!} +\\
             &&- \sum_{j=1}^s \alpha_j \int_M \frac{c_k(E)^{s-1} c_{k-1}(E) c_1(M)^{n-sk+1}}{(n-sk+1)!} \\
d_0(X,L|_X) &=& \int_M \frac{c_k(E)^s c_1(M)^{n-sk}}{(n-sk)!}. 
\end{eqnarray*}
\end{thm}

\begin{rem} 
Clearly the linearization of $E$ is fixed from the one of $L$ thanks to the hypothesis $(\det E)^q \simeq L^p$ as linearized bundles. The latter is crucial to get the compact formula \eqref{eq::F_E_j}. Indeed $\alpha_j$ and $a_0(X,L|_X)$ depend on the linearization of $E$ and $L$ respectively, but on the other hand $F(X,L|_X, \rho)$ is independent of the linearization of $L$.
\end{rem}

\begin{proof}[proof of theorem \ref{thm::main_vect}]
Since $c_{sk}(E^{\oplus s})=c_k(E)^s$, $c_1(E^{\oplus s}) = s\, c_1(E)$ and  by hypothesis $q\,c_1(E)=c_1(L^p)=p\,c_1(M)$, by lemma \ref{lem::asym_h^0_vect} we get
\begin{eqnarray*}
d_0(X) &=& \int_M \frac{c_k(E)^s c_1(M)^{n-sk}}{(n-sk)!} \\
d_1(X) &=& (1-\frac{p}{q}s)\int_M \frac{c_k(E)^s c_1(M)^{n-sk}}{2(n-sk-1)!} = \frac{(q-ps)(n-sk)}{2q}d_0(X).
\end{eqnarray*}
Since $F(X,L|_X, \rho)$ is indipendent of the linearization on $L$, we are free to change it to make easier the calculations. In particular we choose on $L\simeq K_M^{-1}$ the natural linearization coming from the lifting of the $\mathbb C^\times$-action on the holomorphic tangent bundle $TM$. This gives $c_1^G(L)=c_1^G(M)$, wehere $c_1^G$ denote the equivariant first chern class (in the Cartan model of the equivariant cohomology of $M$). To preserve the hypothesis we have to vary accordingly the linearization of $E$ to have $q\,c^G_1(E)=c^G_1(L^p)=p\,c^G_1(M)$.       
Finally, by relations $c^G_{sk}(E^{\oplus s})=c^G_k(E)^s$, $c^G_1(E^{\oplus s}) = s\, c^G_1(E)$ and lemma \ref{lem::asym_w^0_vect} we have
\begin{eqnarray*}
a_0(X) &=& \int_M \frac{c_k^G(E)^s c_1^G(M)^{n-sk+1}}{(n-sk+1)!} - \sum_{j=1}^s \alpha_j \int_M \frac{c_k(E)^{s-1} c_{k-1}(E) c_1(M)^{n-sk+1}}{(n-sk+1)!} \\
a_1(X) &=& (1-\frac{p}{q}s)\int_M \frac{c_k^G(E)^s c_1^G(M)^{n-sk+1}}{2(n-sk)!} + \sum_{j=1}^s k\alpha_j \int_M \frac{c_k(E)^s c_1(M)^{n-sk}}{2(n-sk)!} + \\
&&- (1-\frac{p}{q}s)\sum_{j=1}^s \alpha_j \int_M \frac{c_k(E)^{s-1} c_{k-1}(E) c_1(M)^{n-sk+1}}{2(n-sk)!} \\
&=& \frac{(q-ps)(n-sk+1)}{2q} a_0(X,L_X) + \frac{k}{2} d_0(X,L_X) \sum_{j=1}^s \alpha_j.
\end{eqnarray*}
Thus, by definition \ref{defn::DF_invariant} we get
\begin{eqnarray*}
F(X,L|_X, \rho) &=& \frac{(q-ps)(n-sk)}{2q} \frac{a_0(X,L|_X)}{d_0(X,L|_X)} \\
& & - \frac{(q-ps)(n-sk+1)}{2q} \frac{a_0(X,L|_X)}{d_0(X,L|_X)} - \frac{k}{2} \sum_{j=1}^s \alpha_j \\
&=& \frac{ps-q}{2q} \frac{a_0(X,L|_X)}{d_0(X,L|_X)} - \frac{k}{2} \sum_{j=1}^s \alpha_j.
\end{eqnarray*}
\end{proof}

\bigskip
When $E$ has the right linearization, the Futaki invariant of $X$ is a multiple of the weight $\sum_{j=1}^s\alpha_j$ of $P= {\rm span}\{\sigma_1,\dots,\sigma_s\}$. Indeed we have the following 

\begin{cor}\label{cor::Futaki_peso}
In the situation of theorem \ref{thm::main_vect}, if the choosen linearization on $E$ satisfies $\int_M c_k^G(E)^s c_1^G(E)^{n-sk+1} = 0$ then \begin{equation}\label{eq::F_mult_MW} F(X,L|_X, \rho) = -C T \sum_{j=1}^s \alpha_j,\end{equation} where $C=\left(2p(n-sk+1)\int_M c_k(E)^sc_1(E)^{n-sk}\right)^{-1} >0$ and 
\begin{multline*} T = kp(n-sk+1)\int_M c_k(E)^sc_1(E)^{n-sk} \\ - (q-ps) \int_M c_k(E)^{s-1}c_{k-1}(E)c_1(E)^{n-sk+1}. \end{multline*} 
\end{cor}

\begin{proof}
Substituting the expressions of $a_0(X,L|_X)$ and $d_0(X,L|_X)$ on \eqref{eq::F_E_j} we get 
\begin{equation*}\label{eq::F_E_j_exp} 
F(X,L|_X, \rho) = - C \left( (q-ps)\int_M c_k^G(E)^s c_1^G(E)^{n-sk+1} + T \sum_{j=1}^s \alpha_j \right), \end{equation*} and formula \eqref{eq::F_mult_MW} follows immediatly by hypothesis.

To show the positivity of the constant $C$ is enough to observe that $L_X$ is ample and, by definition of $d_0(X,L|_X)$, the constant $1/C$ is a positive multiple of the degree of $(X,L|_X)$. 
\end{proof}

\begin{rem}\label{rem::positivity_T} Establishing the positivity of the constant $T$ is a problem quite delicate. At least when $E$ is ample, one would apply the theory of Fulton and Lazarsfeld \cite{FL83} to conclude that $T>0$. This is true when $q-ps \leq 0$ (i.e., by adjunction formula, when $X$ is not Fano), but unfotunately this is not true in general because the polynomial in the Chern classes defining $T$ is not numerically positive. Nevertheless, if $E$ is very ample (i.e. the tautological line bundle $\mathcal O _{\mathbb P(E)}(1)$ on $\mathbb P(E)$ is very ample), then by a theorem of of Beltrametti, Schneider and Sommese \cite{BSS96} we get the bound $$ T \geq k^{n-sk+1} (p(n+1) - kq), $$ that 
already gives a good number of examples, some of which are described in the last section.  
\end{rem}


\section{The case $E_j \simeq L^{r_j}$}\label{sec::Lrj}

Now we turn to consider the second case mentioned in the introduction. In particular we allow 
$L \neq K^{-1}_M$, but we consider sections $\sigma_j \in H^0(M,L^{r_j})$ in some tensor 
power of the polarizing bundle $L$. We have the following

\begin{thm}\label{thm::main_lin}
Let $(M,L)$ be a $n$-dimensional polarized manifold endowed with a $\mathbb C^\times$-action 
$\rho: \mathbb C^\times \to {\rm Aut}(M)$ and a linearization on $L$. 
For each $j\in\{1,\dots,s\}$ consider a section $\sigma_j \in H^0(M,L^{r_j})$ such that 
$\rho(t) \cdot \sigma_j = t^{\alpha_j} \sigma_j$ for some $\alpha_j \in \mathbb Z$. Let 
$X=\bigcap_{j=1}^s \sigma^{-1}(0)$. Suppose $\dim(X) = n-s$, then we have

\[
\begin{array}{cl} \label{eq::DF_CI}
F(X, L|_X, \rho)&  = F(M, L, \rho)  +\\
& \frac{1}{2} \left( - \sum_{j=1}^s \left( \frac{\alpha_j}{r_j} - \frac{a_0}{d_0} \right) r_j + 
\frac{\frac{2d_1}{nd_0} - \sum_{j=1}^s r_j}{n+1-s} \sum_{j=1}^s \left( \frac{\alpha_j}{r_j} - 
\frac{a_0}{d_0} \right)\right),
\end{array}
\]
\noindent where $a_0=a_0(M,L)=\int_M \frac{c_1^G(L)^{n+1}}{(n+1)!}$, 
$d_0 = d_0(M,L)= \int_M \frac{c_1(L)^n}{n!}$ and 
$d_1= d_1(M,L)=\int_M \frac{c_1(L)^{n-1}c_1(M)}{2(n-1)!}$.\end{thm}

\begin{proof}
Since $c_s(L^{r_1}\oplus \dots \oplus L^{r_s})=\left(\prod_{j=1}^s r_j\right) c_1(L)^s$, 
$c_1(L^{r_1}\oplus \dots \oplus L^{r_s}) = \sum_{j=1}^s r_j c_1(L)$, by 
Lemma \ref{lem::asym_h^0_vect} we get
\begin{eqnarray*}
d_0(X) &=& \left(\prod_{j=1}^s r_j\right) \int_M \frac{c_1(L)^n}{(n-s)!} = 
\frac{ d_0 \, n!}{(n-s)!} \prod_{j=1}^s r_j \\
d_1(X) &=& \left(\prod_{j=1}^s r_j\right) \int_M \frac{ \left(c_1(M) - 
\sum_{j=1}^s r_j c_1(L)\right) c_1(L)^{n-1}}{2(n-s-1)!} \\
&=& \left(\frac{2d_1}{nd_0} - \sum_{j=1}^s r_j\right)
\frac{d_0 \, n!}{2(n-s-1)!}\prod_{j=1}^s r_j = \frac{n-s}{2}\left(\frac{2d_1}{nd_0} - 
\sum_{j=1}^s r_j\right) d_0(X),
\end{eqnarray*}
and analogously by \ref{lem::asym_w^0_vect}
\begin{eqnarray*}
a_0(X) &=& \left( \int_M \frac{c_1^G(L)^{n+1}}{(n-s+1)!} - 
\sum_{j=1}^s \frac{\alpha_j}{r_j} \int_M \frac{c_1(L)^n}{(n-s+1)!}\right) 
\prod_{j=1}^s r_j \\
&=& \frac{ a_0 (n+1)!}{(n-s+1)!} \prod_{j=1}^s r_j - \frac{\sum_{j=1}^s 
\frac{\alpha_j}{r_j}}{n-s+1}d_0(X)\\
a_1(X) &=& \left( \int_M  \frac{\left(c_1^G(M) - \sum_{j=1}^s r_j c_1^G(L)\right) 
c_1^G(L)^{n}}{2(n-s)!} + \sum_{j=1}^s \alpha_j \int_M \frac{c_1(L)^{n}}{2(n-s)!} + \right. \\
&&\left. - \sum_{j=1}^s \frac{\alpha_j}{r_j} \int_M \frac{\left(c_1(M) - 
\sum_{j=1}^s r_j c_1(L)\right) c_1(L)^{n-1}}{2(n-s)!}\right) \prod_{j=1}^s r_j \\
&=& \frac{a_1 n!}{(n-s)!} \prod_{j=1}^s r_j - \frac{a_0 (n+1)! \sum_{j=1}^s r_j}{2(n-s)!} 
\prod_{j=1}^s r_j + \\
&& +\frac{1}{2} d_0(X) \sum_{j=1}^s \alpha_j - \frac{1}{n-s} d_1(X) 
\sum_{j=1}^s \frac{\alpha_j}{r_j},
\end{eqnarray*}
Thus $F(X,L|_X, \rho)$ equals to 
\[
\begin{array}{ll}
 \left( \frac{n+1}{n-s+1} \frac{a_0}{d_0}  - \frac{\sum_{j=1}^s 
\frac{\alpha_j}{r_j}}{n-s+1}\right)\frac{d_1(X)}{d_0(X)} 
- \frac{a_1}{d_0} + 
 \frac{n+1}{2} \frac{a_0}{d_0} \sum_{j=1}^s r_j 
-\frac{1}{2} \sum_{j=1}^s \alpha_j + &\\ 
+ \frac{1}{n-s} \frac{d_1(X)}{d_0(X)} 
\sum_{j=1}^s \frac{\alpha_j}{r_j} 
 = \left( \frac{a_0}{d_0} +\frac{s}{n-s+1} \frac{a_0}{d_0}  + \frac{\sum_{j=1}^s 
\frac{\alpha_j}{r_j}}{(n-s)(n-s+1)}\right)\frac{d_1(X)}{d_0(X)} - & \\
- \frac{a_1}{d_0} + \frac{n+1}{2} \frac{a_0}{d_0} \sum_{j=1}^s r_j -\frac{1}{2} \sum_{j=1}^s \alpha_j 
= \frac{a_0d_1}{d_0^2} - \frac{n}{2} \sum_{j=1}^s r_j - & \\
- \frac{s}{2} 
\left(\frac{2d_1}{nd_0} - \sum_{j=1}^s r_j\right) \frac{a_0}{d_0} +\ 
 \frac{s}{2}
\frac{n-s}{n-s+1} \left(\frac{2d_1}{nd_0} - \sum_{j=1}^s r_j\right) \frac{a_0}{d_0} + &\\
+ \frac{\frac{1}{2}\left(\frac{2d_1}{nd_0} - \sum_{j=1}^s r_j\right)\sum_{j=1}^s 
\frac{\alpha_j}{r_j}}{n-s+1} 
- \frac{a_1}{d_0}  +  \frac{n+1}{2} \frac{a_0}{d_0} 
\sum_{j=1}^s r_j -\frac{1}{2} 
\sum_{j=1}^s \alpha_j  =& \\
=  \frac{a_0d_1}{d_0^2} - \frac{a_1}{d_0} - \frac{\frac{1}{2} \left(\frac{2d_1}{nd_0} - 
\sum_{j=1}^s r_j\right)s\frac{a_0}{d_0}}{n-s+1} + \\ 
 + \frac{\frac{1}{2}\left(\frac{2d_1}{nd_0} - \sum_{j=1}^s r_j\right)\sum_{j=1}^s 
\frac{\alpha_j}{r_j}}{n-s+1}
 + \frac{1}{2} \frac{a_0}{d_0} \sum_{j=1}^s r_j -\frac{1}{2} \sum_{j=1}^s \alpha_j \\
= \frac{a_0d_1}{d_0^2} - \frac{a_1}{d_0} + \frac{1}{2}\frac{\frac{2d_1}{nd_0} - 
\sum_{j=1}^s r_j}{n-s+1} \sum_{j=1}^s \left( \frac{\alpha_j}{r_j} - \frac{a_0}{d_0}\right) 
- \frac{1}{2} \sum_{j=1}^s \left( \frac{\alpha_j}{r_j} - \frac{a_0}{d_0} \right) r_j &
\end{array}
\]
and we are done.
\end{proof}

When $M=\mathbb P^n$ and $L=\mathcal O_{\mathbb P^n}(1)$ theorem \ref{thm::main_lin} gives the following result due to Z. Lu on complete intersections \cite{Lu}:

\begin{cor}
Let $X \subset \mathbb C^n$ be a $(n-s)$-dimensional subvariety defined by homogeneous polynomials $F_1, \dots, F_s$ of degree $r_1,\dots r_s$ respectively. Let $\rho:\mathbb C^\times \to SL(n+1)$ be a one parameter subgroup such that $$ \rho(t) \cdot F_j = t^{\alpha_j} F_j, \qquad j=1,\dots,s$$ for some $\alpha_1,\dots,\alpha_s \in \mathbb Z$.
Then we have
$$
F\left(X,\mathcal O_X(1), \rho\right) = \frac{1}{2} \left( - \sum_{j=1}^s \alpha_j + \frac{n+1 - \sum_{j=1}^s r_j}{n+1-s} \sum_{j=1}^s \frac{\alpha_j}{r_j} \right).
$$
\end{cor}

\begin{proof}
Since $H^0(\mathbb P^n, \mathcal O_{\mathbb P^n}(m)) \simeq \mathbb C[z_0,\dots,z_n]_m$ then $$h^0(\mathbb P^n, \mathcal O_{\mathbb P^n}(m)) = \binom{n+m}{m} = \frac{1}{n!} m^n + \frac{n(n+1)}{2n!} m^{n-1} + O(m^{n-2}),$$ thus $\frac{2d_1}{nd_0} = n+1$. Moreover, taking on $\mathcal O_{\mathbb P^n}$ the unique linearization induced by $SL(n+1)$ we get $w(\mathbb P^n, \mathcal O_{\mathbb P^n}(m))=0$, and in particular $a_0=0$.
\end{proof}

The formula \eqref{eq::DF_CI} becomes simpler if all the $r_j$'s are equal. Moreover in this case $F(X,L|_X, \rho)$ has a nice expression in term of the so-called ``Chow weight'' of $(X,L|_X)$, whose definition, essentially due to Mumford \cite{Mum77}, is the following

\begin{defn}\label{defn::Cw}
In the situation of definition \ref{defn::DF_invariant}, let $X \subset V$ be a $s$-codimensional invariant subvariety. Thus $L|_X$ is a linearized line bundle and we have the asymptotic expansions 
\begin{eqnarray*} 
w(X,L|_X^m) &=& a_0(X,L|_X) m^{n-s+1} + O(m^{n-s}) \\ h^0(X,L|_X^m) &=& 
d_0(X,L|_X) m^{n-s} + O(m^{n-s-1}). 
\end{eqnarray*} 
The \emph {Chow weight} of $X$ with respect the chosen one-parameter subgroup of ${\rm Aut}(V)$ is 
$$ \mu(X,V,L) = \frac{a_0(V,L)}{d_0(V,L)} - \frac{a_0(X,L|_X)}{d_0(X,L|_X)}.$$
If $G \subset {\rm Aut}(V)$ is a reductive subgroup, we say that $X$ is \emph{Chow stable (resp. semi-stable) w.r.t. G} if $ \mu(X) < 0$ (resp. $\leq$) for all one-parameter subgroups of $G$.
\end{defn}

\begin{cor}\label{cor::eq_deg}
In the situation of theorem \ref{thm::main_lin}, if $r_j=r$ for all $j$, then $X$ is a section of $M$ determined by the linear system $P = {\rm span} (\sigma_1, \dots \sigma_s) \subset H^0(M,L^r)$. 
In this case we have: 
\begin{eqnarray*} F(X,L|_X, \rho) &=& F(M,L, \rho)  +
 \frac{\frac{2d_1}{nd_0}-r(n+1)}{2(n+1-s)} \sum_{j=1}^s 
\left(\frac{\alpha_j}{r} - \frac{a_0}{d_0}\right) \\ &=& F(M,L, \rho) - C \mu \left(X,M,L\right), 
\end{eqnarray*} 
where $C=r\frac{n+1}{2}-\frac{d_1(M,L)}{nd_0(M,L)} \geq 0$ with equality if and only if $M\simeq \mathbb P^n$ and $r=1$.
\end{cor}

\begin{proof}
The first equation is an obvious consequence of \eqref{eq::DF_CI} when $r_j = r$ for all $j$. The second is a consequence of $$ \mu(X,M,L) = \frac{1}{n+1-s} \sum_{j=1}^s \left( \frac{\alpha_j}{r_j} - \frac{a_0}{d_0}\right),$$ which follows from definition of Chow weight \ref{defn::Cw} and from formulae for $a_0(X,L|_X)$ and $d_0(X,L|_X)$ in lemmata \ref{lem::asym_h^0_vect} and \ref{lem::asym_w^0_vect}. To prove the non-negativity of $C$, by Kobayashi-Ochiai \cite[Theorem 1.1]{KO73} we get $c_1(M) \leq (n+1)c_1(L)$ with equality if and only if $(M,L) = (\mathbb P^n,\mathcal O_{\mathbb P^n}(1))$. Thus we have $$ \frac{d_1(M,L)}{n\,d_0(M,L)} = \frac{\int_Mc_1(L)^{n-1}c_1(M)}{2\int_Mc_1(L)^n} \leq \frac{n+1}{2}$$ and the statement follows.
\end{proof}

\begin{rem}
 In the case $F(M,L, \rho)=0$ (e.g. when $M$ admits a cscK metric in $c_1(L)$), the sign of 
 $F(X,L|_X,\rho)$ is the product of the signs of $\mu(X,M,L)$. Thus $X$ is $K$-unstable if it is Chow-unstable. 
\end{rem}


\section{Proofs of fundamental lemmata}

\begin{lem}\label{lem::id_chern}
Let $B$ be a holomorphic vector bundle of rank $b$ on a manifold $M$, then $$\sum_{p=0}^k (-1)^p ch\left(\wedge^p B^*\right) = c_b(B) td(B)^{-1}.$$
\end{lem}

\begin{proof}
It is lemma 18 in \cite{BS58} (see also \cite[Example 3.2.5]{F}).
Let $\alpha_1,\dots,\alpha_b$ be Chern roots of $B$. Since $ch\left( \wedge^p B^* \right) = \sum_{1 \leq i_1 < \dots < i_p \leq b} e^{-(\alpha_{i_1}+\dots+\alpha_{i_p})}$, then we have \begin{eqnarray*}
\sum_{p=0}^k (-1)^pch\left(\wedge^p B^*\right) &=& \sum_{p=0}^b (-1)^p \sum_{1 \leq i_1 < \dots < i_p \leq b} e^{-(\alpha_{i_1}+\dots+\alpha_{i_p})} \\
&=& \prod_{i=1}^b \left( 1-e^{-\alpha_i} \right) \\
&=& \prod_{i=1}^b \alpha_i \prod_{i=1}^b \frac{ 1 - e^{-\alpha_i} }{\alpha_i},
\end{eqnarray*}
and the statement is proved.
\end{proof}

\begin{lem}\label{lem::asym_h^0_vect}
Let $(M,L)$ be a $n$-dimensional polarized manifold and let $E_1,\dots,E_s$ be a collection of holomorphic vector bundles on $M$. Set $k_j= {\rm rank} (E_j)$, 
$B=E_1\oplus \dots \oplus E_s$ and $b={\rm rank}(B)=\sum_{j=1}^s k_j$.  
For each $j\in\{1,\dots,s\}$ consider a non-zero section $\sigma_j \in H^0(M,E_j)$ 
and set $\sigma = (\sigma_1,\dots,\sigma_s) \in H^0(M,B)$ and $X=\sigma^{-1}(0)$. 
If $\dim(X)=n-b$ we have the asymptotic expansion as 
$k \to +\infty$ $$ h^0 \left( X , L|_X^m\right) = d_0(X) m^{n-b} + d_1(X) m^{n-b-1} + O(m^{n-b-2}),$$ where
\begin{eqnarray*}
d_0(X) &=& \int_M \frac{c_b(B) c_1(L)^{n-b}}{(n-b)!} \\
d_1(X) &=& \int_M \frac{c_b(B) \left(c_1(M) - c_1(B)\right) c_1(L)^{n-b-1}}{2(n-b-1)!},
\end{eqnarray*}
(here $c_b(B) = \prod_{j=1}^s c_{k_j}(E_j)$ and $c_1(B)=\sum_{j=1}^s c_1(E_j)$).
\end{lem}

\begin{proof}
Let $\mathcal O_X$ be the structure sheaf of $X$. By assumption $\sigma$ is a regular section, so the Koszul complex 
$$0 \to \wedge^b B^* \to \wedge^{b-1} B^* \to \dots \to B^* \to \mathcal O_M \to \mathcal O_X \to 0 $$ induced by $\sigma$ is exact. 
Tensoring by $L^m$ preserves the exacteness, thus 
$$\chi(X,L|_X^m ) = \sum_{p=0}^b (-1)^p \chi\left(M,L^m \otimes \wedge^p B^*\right)$$ 
and by the Hirzebruch-Riemann-Roch theorem we get

\begin{eqnarray*}
\chi(X,L|_X^m) &=& \sum_{p=0}^b (-1)^p \int_M ch\left( \wedge^p B^* \right) e^{mc_1(L)} td(M) \\
&=& \int_M ch\left( \sum_{p=0}^b (-1)^p \wedge^p B^* \right) e^{m c_1(L)} td(M) \\
&=& \int_M c_b(B) td(B)^{-1} e^{m c_1(L)} td(M),
\end{eqnarray*}
where second equality follows by elementary properties of the Chern character, and the last one holds by lemma \ref{lem::id_chern}.

As $k\to +\infty$ we have the expansion

\begin{eqnarray*}
\chi(X,L|_X^m) &=& m^{n-b} \int_M \frac{c_b(B) c_1(L)^{n-b}}{(n-b)!} + \\
&& + m^{n-b-1}  \int_M \frac{c_b(B) \left(c_1(M) - c_1(B)\right) c_1(L)^{n-b-1}}{2(n-b-1)!} \\
&& + O(m^{n-b-2}),
\end{eqnarray*}
where we used $td(M)=1+\frac{1}{2}c_1(M)+\dots$ and $td(B)^{-1} = 1-\frac{1}{2}c_1(B) + \dots$ (dots representing terms of degree greater then one). Finally the equality $h^0(X,L|_X^m) = \chi(X,L|_X^m)$ follows by ampleness of $L$.
\end{proof}

\begin{lem}\label{lem::asym_w^0_vect}
Let $(M,L)$ be a $n$-dimensional polarized manifold endowed with a $\mathbb C^\times$-action $\rho: \mathbb C^\times \to {\rm Aut}(M)$ and a linearization on $L$. Let $E_1,\dots,E_s$ be a collection of linearized vector bundles on $M$. Set $k_j= {\rm rank} (E_j)$, $B=E_1\oplus \dots \oplus E_s$ and $b={\rm rank} (B)=\sum_{j=1}^s k_j$.  For each $j\in\{1,\dots,s\}$ consider a non-zero section $\sigma_j \in H^0(M,E_j)$ such that $\rho(t) \cdot \sigma_j = t^{\alpha_j} \sigma_j$ for some $\alpha_j \in \mathbb Z$, and set $\sigma = (\sigma_1,\dots,\sigma_s) \in H^0(M,B)$ and $X=\sigma^{-1}(0)$. If $\dim(X)=n-b$, then we have the asymptotic expansion as $k \to +\infty$ $$ w^0 \left( X , L|_X^m\right) = a_0(X) m^{n-b+1} + a_1(X) m^{n-b} + O(m^{n-b-1}),$$ where
\begin{eqnarray*}
a_0(X) &=& \int_M \frac{c_b^G(B) c_1^G(L)^{n-b+1}}{(n-b+1)!} - \sum_{j=1}^s \alpha_j \int_M \frac{c_b(B) c_{k_j-1}(E_j) c_1(L)^{n-b+1}}{(n-b+1)!\, c_{k_j}(E_j)} \\
a_1(X) &=& \int_M \frac{c_b^G(B) \left(c_1^G(M) - c_1^G(B)\right) c_1^G(L)^{n-b}}{2(n-b)!}   \\
&& + \sum_{j=1}^s k_j\alpha_j \int_M \frac{c_b(B) c_1(L)^{n-b}}{2(n-b)!}  \\
&&- \sum_{j=1}^s \alpha_j \int_M \frac{c_b(B) c_{k_j-1}(E_j) \left(c_1(M) - c_1(B)\right) c_1(L)^{n-b}}{2(n-b)!\, c_{k_j}(E_j)},
\end{eqnarray*}
(here $c_b^G(B) = \prod_{j=1}^s c_{k_j}^G(E_j)$ and $c_1^G(B)=\sum_{j=1}^s c_1^G(E_j)$).
\end{lem}

\begin{proof}
It is very similar to the previous on the dimension of $H^0(X,L|_X^m)$. Since sections $\sigma_j$ are only semi-invariant, they do not give rise to equivariant sequences of bundles, but to overcame the problem we can initially change the linearization of each $E_j$ and go back to original one at the end of computations. Denoted by $\mathbb C_{\beta}$ the trivial line bundle on $M$ with linearization $t \cdot u = t^{\beta} u$, for each $j \in \{1,\dots,s\}$ let $$ F_j = E_j \otimes \mathbb C_{-\alpha_j}.$$ In this way, each $\sigma_j \in H^0(M,E_j)$ is an invariant section.

Now consider the rank $b = \sum_{j=1}^s k_j$, $\mathbb C^\times$-linearized vector bundle $F=\bigoplus_{j=1}^s F_j$, and let $\sigma \in H^0(M,F)$ be the holomorphic section defined by $\sigma= (\sigma_1,\dots,\sigma_s)$. Clearly $\sigma$ is invariant and we have $X = \sigma^{-1}(0)$. Let $\mathcal O_X$ be the structure sheaf of $X$. By assumption $\sigma$ is a regular section, so the Koszul complex $$0 \to \wedge^b F^* \to \wedge^{b-1} F^* \to \dots \to F^* \to \mathcal O_M \to \mathcal O_X \to 0 $$ induced by $\sigma$ is exact and equivariant. Tensoring by $L^m$ preserves the exacteness and equivariance, thus $$\chi^G(X,L^m|_X) = \sum_q (-1)^q {\rm tr}\left(e^{it} | H^q(X, L^m|_X) \right) = \sum_{p=0}^b (-1)^p \chi^G\left(M,L^m \otimes \wedge^p F^*\right)$$ and by the equivariant Riemann-Roch theorem we get

\begin{eqnarray*}
\chi^G(X,L|_X^m) &=& \sum_{p=0}^b (-1)^p \int_M ch^G\left( \wedge^p F^* \right) e^{mc_1^G(L)} td^G(M) \\
&=& \int_M ch^G\left( \sum_{p=0}^b (-1)^p \wedge^p F^* \right) e^{m c_1^G(L)} td^G(M) \\
&=& \int_M c_b^G(F) td^G(F)^{-1} e^{m c_1^G(L)} td^G(M),
\end{eqnarray*}
where the last equality holds by lemma \ref{lem::id_chern}. Since the right part of the equivariant Riemann-Roch theorem is a power series convergent in some neighborhood of zero of the lie algebra of the acting group, to get the trace of the generator of the action on the virtual space $\bigoplus_q (-1)^q H^q(X,L|_X^m)$, is sufficient to take the ``linear term'' of the integrand. Explicitly, as $m\to +\infty$ we have $H^q(X,L|_X^m) = 0$ for $q>0$ by ampleness of $L$, and we get the expansion
\[
\begin{array}{ccl}\label{eq::trace_shifed_lin}
w^0(X,L|_X^m) &=& m^{n-b+1} \int_M \frac{c^G_b(F) c^G_1(L)^{n-b+1}}{(n-b+1)!} + \\
&& + m^{n-b}  \int_M \frac{c^G_b(F) \left(c^G_1(M) - c^G_1(F)\right) c^G_1(L)^{n-b}}{2(n-b)!} 
+ O(m^{n-b-1}),
\end{array}
\]
where we used $td(M)=1+\frac{1}{2}c_1(M)+\dots$ and $td(F)^{-1} = 1-\frac{1}{2}c_1(F) + \dots$ (dots representing terms of degree greater then one). Finally we have to come back to original linearization of $E_j$'s. Since $F_i = E_i \otimes \mathbb C_{-\alpha_i}$, by the Cartan model of the equivariant cohomology of $M$, is easy to see that $c_1^G(F_j) = c_1^G(E_j) - k_j \alpha_j$ and $c_{k_j}^G(F_j) = \sum_{p=0}^{k_j} (-\alpha_j)^{k_j-p} c_p^G(E_j)$, whence
\[
\begin{array}{ccl}
c_k^G(F) &=& \prod_{j=1}^s c_{k_j}^G(F_j) = \prod_{j=1}^s \sum_{p=0}^{k_j} 
(-\alpha_j)^{k_j-p} c_p^G(E_j) \\
& = & c_b^G(B)
\left(1 - \sum_{j=1}^s \alpha_j \frac{c_{k_j-1}^G(E_j)}{c_{k_j}^G(E_j)} + \dots \right) \\
c_1^G(F) &=& \sum_{j=1}^s c_1^G(F_j) = c_1^G(E) - \sum_{j=1}^s k_j \alpha_j.
\end{array}
\]
and substituting in \eqref{eq::trace_shifed_lin} we are done.
\end{proof}


\section{Applications and examples}\label{sec::appex}

In this section we show some consequences of the Theorems \ref{thm::main_vect} and \ref{thm::main_lin}. In particular we use those theorems to calculate the Futaki invariant of central fibers of test configurations arising from degenerations of linear sections of vector bundles.

More precisely consider a $n$-dimensional polarized manifold $(M,L)$ endowed with a one-parameter subgroup of automorphisms $\rho:\mathbb C^\times \to {\rm Aut}(M)$ that linearizes on $L$. Let $P = {\rm span} (\eta_1, \dots, \eta_s) \subset H^0(M,E)$ be an $s$-dimensional linear system of a rank $k$ linearized holomorphic vector bundle $E$ on $M$. Thus $X_{P} = \bigcap_{j=1}^s \eta_j^{-1}(0)$ is $s$-codimensional. The $\rho$-action on $P$ gives naturally a test configuration for the variety $(X_P,L|_{X_P})$ as follows. Let $P_t=\rho(t) \cdot P$ and let $\mathcal X$ be the closure of $\left\{ (x,t) \in M\times \mathbb C^\times \,|\, x \in X_{P_t} \right\}$ in $ M \times \mathbb C$. The projection on the second factor induces a flat morphism $\pi : \mathcal X \to \mathbb C$. Let $X_{P_0} = \bigcap_{j=1}^s \sigma_j^{-1}(0)$, where $P_0 = {\rm span}(\sigma_1, \dots, \sigma_s) = \lim_{t \to 0} \rho(t) \cdot P$ with $\sigma_j$'s semi-invariant. By the uniqueness \cite[Proposition 9.8]{Har77} we have $\pi^{-1}(0) = X_{P_0}$.

\subsection{A Mukai-Umemura-Tian like example with singular central fibre} 

Consider the grassmannian $M=G(4,6)$ of $4$-planes in $\mathbb C^6$ polarized with $L= \bigwedge^2 Q$, being $Q$ the universal quotient bundle. Since the Kodaira map induced by $L$ is the Pl\"uker embedding $M \hookrightarrow \mathbb P^{14}$, for each $\eta_1, \eta_2, \eta_3 \in H^0(M,L)$ linearly independent, the subvariety $X = \bigcap_{j=1}^3 \eta_j^{-1}(0)$ is a section of $G(4,6)$ with a 3-codimensional subspace in $\mathbb P^{14}$. The general $X$ arising in this way is a Fano 5-fold. 

\noindent Let $\rho:\mathbb C^\times \to SL(6)$ be the one parameter subgroup generated by 
\newline ${\rm diag}(-5,-3,-1,1,3,5)$ and consider $P_\epsilon = {\rm span}\{\eta_1, \eta_2, \eta_3\} \subset H^0(M,L)$ where $$ \eta_1=e_{16}+e_{25}+e_{34},\quad \eta_2=e_{15}+e_{24}+\varepsilon e_{46},\quad \eta_3=e_{26}+e_{35}+\varepsilon e_{45} $$ and we idenfify $H^0(M,L)\simeq \bigwedge^2\mathbb C^6$. 

\noindent By local calculations it is easy to see that $X_{P_0}$ is $\mathbb C^\times$-invariant and is singular at points $e_2 \wedge e_3 \wedge e_5 \wedge e_6$ and $e_1 \wedge e_2 \wedge e_4 \wedge e_5$. On the other hand, for $\varepsilon \neq 0$ the variaty $X_{P_\varepsilon}$ is non-singular but not invariant.

\noindent Now let $$ \sigma_1=e_{16}+e_{25}+e_{34},\quad \sigma_2=e_{15}+e_{24},\quad \sigma_3=e_{26}+e_{35}. $$ We have $P_0 = {\rm span}\{\sigma_1,\sigma_2,\sigma_3\}$, moreover $\rho(t) \cdot P_\varepsilon$ tends to $P_0$ as $t \to 0$. Thus, following the construction shown at the start of this section, there is a test configuration of $(X_{P_\varepsilon},L|_{X_{P_\varepsilon}})$ with central fibre $(X_{P_0},L|_{X_{P_0}})$. Since 
$$ \rho(t) \cdot \sigma_1 = \sigma_1, \quad \rho(t) \cdot \sigma_2 = t^{-2} \sigma_2, \quad \rho(t) \cdot \sigma_3 = t^2 \sigma_3, $$ 
by the Corollary \ref{cor::eq_deg} we get 
$$F(X_{P_0},L|_{X_{P_0}}, \rho) = C\, (0-2+2) = 0, $$
 where we used  $F(M,L,\rho)=0$ and $a_0(M,L)=0$. 

Hence by \cite{T97} or \cite{D02} we proved the following 

\begin{prop}
For each $\varepsilon \neq 0$ the manifold $X_{P_\varepsilon}$ is not $K$-stable, hence is not K\"ahler-Einstein.  
\end{prop} 

\subsection{The quintic Del Pezzo threefold}\label{subsec::X_5} 

Consider the Grassmannian $M=G(2,5)$ of planes in $\mathbb C^5$ polarized with 
$L=\bigwedge^3 Q$, where $Q$ is the universal quotient bundle. As well known the Kodaira map induced by $L$ is the Pl\"uker embedding $ M \hookrightarrow \mathbb P^9 $. Thus for each $\sigma_1, \sigma_2, \sigma_3 \in H^0(M,L)$ linearly independent, the subvariety $X = \bigcap_{j=1}^3 \sigma_j^{-1}(0)$ is a section of $G(2,5)$ with a 3-codimensional subspace in $\mathbb P^9$. The general $X$ arising in this is the quintic Del Pezzo threefold \cite{IP}, in particular it is Fano.

\begin{prop}\label{prop::stabX5}
Each degeneration of $X$ induced by a one-parameter subgroup $\rho : \mathbb C^\times \to {\rm Aut}(M)$ has non-negative Futaki invariant.
\end{prop}

\begin{proof}
Consider the isomorphism $H^0(M,L)\simeq \mathbb \bigwedge^3 \mathbb C^5$ given by $ \bigwedge^3 \mathbb C^5 \ni v \mapsto \sigma_v \in H^0(M,L) $ where $$ \sigma_v(E) = v + \bigwedge^2 \mathbb C^5 \wedge E \in \bigwedge^3 \left(\mathbb C^5 /E \right), $$ for all $E \in M$. Thus we can identify $\sigma_j$ with $u_j \in \bigwedge^3 \mathbb C^5$.

We recall that each automorphism of $M$ comes from the action of an element of $SL(5)$ on $\mathbb C^5$. Thus we can consider $\rho : \mathbb C^\times \to SL(5)$. Let $(e_1,\dots,e_5)$ a basis of eigenvectors and let $\nu_1,\dots,\nu_5 \in \mathbb Z$ be the weights of $\rho$. We have $\nu_1 + \dots + \nu_5 = 0$ and we can suppose without loss $\nu_1 \leq \dots \leq \nu_5$.

Now, since $u_j$'s are general we can also suppose
\begin{eqnarray*}
u_1 &=& \sum_{1\leq i<j<k \leq 5} {c_1^{ijk} e_{ijk}} \\
u_2 &=& \sum_{1\leq i<j<k \leq 5,\, i+j+k \geq 7} {c_2^{ijk} e_{ijk}} \\
u_3 &=& \sum_{1\leq i<j<k \leq 5,\, i+j+k \geq 8} {c_3^{ijk} e_{ijk}}
\end{eqnarray*}
and $c_\ell^{ijk} \neq 0$.

The action induced by $\rho$ on $\bigwedge^3 \mathbb C^5$ gives a weak order ($\preceq$) on the basis $(e_{ijk}\,|\,1 \leq i<j<k \leq 5)$ as follows: we define $e_{i_1j_1k_1} \preceq e_{i_2,j_2,k_2}$ if $\nu_{i_1}+\nu_{j_1}+\nu_{k_1} \leq \nu_{i_2}+\nu_{j_2}+\nu_{k_2}$. Obviously $e_{123} \preceq e_{124} \preceq e_{134} \preceq e_{234}$ and $e_{125} \preceq e_{ij5}$ for all $i<j$. Thus ${\rm span}(u_1, u_2, u_3)$ tends to ${\rm span}(v_1,v_2,v_3)$ under the action of $\rho(t)$ as $t \to 0$, where \begin{eqnarray*}
v_1 &=& \min \{ e_{123} , e_{125}\} = e_{123} \\
v_2 &=& \min \{ e_{124} , e_{125}\} = e_{124} \\
v_3 &=& \min \{ e_{134} , e_{125}\}
\end{eqnarray*}
Let $\alpha_j$ the weight of $v_j$. We have: $\alpha_1=\nu_1 + \nu_2 + \nu_3$, $\alpha_2 = \nu_1 + \nu_2 + \nu_4$ and $\alpha_3= \min \{\nu_1+\nu_3+\nu_4, \nu_1 + \nu_2 + \nu_5\}$. In both cases that can occur is easy to check that $$\alpha_1 + \alpha_2 + \alpha_3 \leq 0.$$

Finally let $X_0 = \lim_{t \to 0} \rho(t) \cdot X = \bigcap_{j=1}^3 \sigma_{v_j}^{-1}(0)$. 
By corollary \ref{cor::eq_deg} we get 
$$ F(X_0,L|_{X_0}, \rho) = -\frac{1}{4} (\alpha_1 + \alpha_2 + \alpha_3) \geq 0,$$ 
where we used that $F(M,L, \rho)=0$, $a_0(M,L)=0$ and $\frac{2\,d_1}{n\,d_0} = 5$ 
(the latter follows from the general fact that if $L^q=K_M^{-1}$ then 
$\frac{2\,d_1(M,L)}{n\,d_0(M,L)}=q$). 
\end{proof}

The result above is an evidence to the $K$-stability of the quintic Del Pezzo threefold $X_5$. In this specific case the above discussion can be strenghtned by observing: 
\begin{itemize}
\item The complex structure of $X = X_5$ is rigid \cite[Corollary 3.4.2]{IP}, hence it cannot be used as central fiber of a test configuration. 
\item It is not hard to adapt Donaldson proof of the existence of K\"ahler-Einstein metric on the Mukai-Umemura manifold $X_{22}$ to this case, hence proving that $X$ is indeed K\"ahler-Einstein and so 
$K$-stable. As showed in \cite{MU}, the manifolds $X_{22}$ and $X_5$ share all the properties involved 
in his argument. In particular we observe that $X_5$ has a $PSL(2)$-invariant anti-canonical section with at worst cusp-like singularities.  
\end{itemize}

\subsection{General complete intersections in Grassmannians}

Following a construction given by Tian \cite{T1}, we generalize Proposition \ref{prop::stabX5} to general intersections of some exterior power of the universal quotient bundle on the Grasmannian. 

As will be clear from the proof the generality condition depends on the one-parameter subgroup $\rho$.

\begin{prop}
Let $(M,L)=(G(k,N),K_{G(k,N)}^{-1})$ be the grassmannian of $k$-planes in $\mathbb C^N$ anti-canonically polarized. Suppose $k(N-k)>N-1$ to avoid trivialities. Fix a one-parameter subgroup 
$\rho: \mathbb C^\times \to {\rm Aut}(M)$ and a linearization to $L$. Denoted by $Q$ the universal 
quotient bundle on $M$, let $E=\bigwedge^{\ell} Q$ endowed with a linearization on $E$ such that 
$(\det E)^N \simeq L^{\binom{N-k}{\ell}}$ as linearized bundles. Let $P \subset H^0(M,E)$ be a 
general $d$-dimensional subspace such that $X_P=\bigcap_{\sigma \in P} \sigma^{-1}(0)$ has dimension $k(N-k)-d\binom{N-k}{\ell} > 0$.

$X_P$ is Fano if and only if $ N-d\binom{N-k}{\ell}>0$. In this case we have 
$$F(X_{P_0},L|_{X_{P_0}}, \rho) > 0,$$ where $P_0=\lim_{t\to 0} \rho(t) \cdot P$.
\end{prop}

\begin{proof}
Take on $E$ and $L$ the unique linearizations induced by $SL(N)$.
Consider the induced representation of $\rho$ on $H^0(M,E)$ and fix a basis of semi-invariant sections $\sigma_1,\dots,\sigma_{h^0(E)}$. Thus for each $j\in \{1,\dots,h^0(E)\}$ there is a unique $\alpha_j\in \mathbb Z$ such that $t \cdot \sigma_j = t^{\alpha_j} \sigma_j$. We can suppose without loss
\begin{equation}\label{eq::ord_basis}
\alpha_i \leq \alpha_j \quad \mbox{if} \quad i <j.
\end{equation}
Let $\eta_1,\dots,\eta_d$ be a basis of $P$. Since $P$ is general we can suppose
\begin{eqnarray*}
\eta_1 &=& \sum_{j=1}^{h^0(E)} c_{1j}\sigma_j\\
\eta_2 &=& \sum_{j=2}^{h^0(E)} c_{2j}\sigma_j\\
       &\vdots&  \\
\eta_d &=& \sum_{j=d}^{h^0(E)} c_{dj}\sigma_j,
\end{eqnarray*}
where $c_{ii}\neq 0$ for all $i\in \{1,\dots,d \}$. Thus the limit of $P$ under the action of $\rho$ is the plane $P_0 = {\rm span}(\sigma_1, \dots , \sigma_d)$. In the chosen linearization $\rho$ acts on $H^0(M,E)$ as a subgroup of $SL(h^0(E))$, thus $\sum_{j=1}^{h^0(E)} \alpha_j = 0$. Hence, by \eqref{eq::ord_basis} and non-triviality of $\rho$ we have \begin{equation}\label{eq::neg_weight}\alpha(P) = \sum_{j=1}^d \alpha_j < 0.\end{equation}

Since $P$ is general, $X_P$ is smooth. Moreover, by the adjunction formula and the hypothesis on $E$ we get $$c_1(X_P)=\iota^*c_1(M)-d\iota^*c_1(E) = \left( 1 - \frac{d}{N} \binom{N-k}{\ell} \right) \iota^*c_1(M),$$ where $\iota : X \hookrightarrow M$ is the inclusion. This prove the Fano condition. 

By the localization theorem for equivariant cohomology is not hard to see that 
$$\int_{G(k,N)} c^G_{\binom{N-k}{\ell}}\left( \bigwedge^\ell Q \right)^d c^G_1
\left( \bigwedge^\ell Q \right)^{k(N-k)-d\binom{N-k}{\ell}+1} = 0.$$ 
Hence, by the Corollary \ref{cor::Futaki_peso} we get 
$$ F(X_{P_0},L|_{X_{P_0}}, \rho) = -CT \sum_{j=1}^d \alpha_j, $$ 
where $C>0$ and 
$$ T > \left(k(N-k)+1-N\right) \binom{N-k}{\ell}^{k(N-k)-d \binom{N-k}{\ell}+2} > 0.$$ 
Actually $\bigwedge^\ell Q$ is not very ample, however in this case we can 
apply \cite[Proposition 1]{BSS96} to get the first inequality above. 
\end{proof}

\end{document}